\documentclass[reqno,12pt, oneside]{amsart}

\usepackage{enumerate}
\usepackage[nobysame]{amsrefs}
\usepackage{enumitem}
\usepackage{amssymb}
\usepackage{amsmath}
\usepackage{amsthm}
\usepackage{microtype}
\usepackage[right=2cm,left=2cm, top=2.5cm, bottom=2.5cm,includehead]{geometry}
\usepackage[mathlines]{lineno}

\newcommand*\patchAmsMathEnvironmentForLineno[1]{%
  \expandafter\let\csname old#1\expandafter\endcsname\csname #1\endcsname
  \expandafter\let\csname oldend#1\expandafter\endcsname\csname end#1\endcsname
  \renewenvironment{#1}%
     {\linenomath\csname old#1\endcsname}%
     {\csname oldend#1\endcsname\endlinenomath}}% 
\newcommand*\patchBothAmsMathEnvironmentsForLineno[1]{%
  \patchAmsMathEnvironmentForLineno{#1}%
  \patchAmsMathEnvironmentForLineno{#1*}}%
\AtBeginDocument{%
\patchBothAmsMathEnvironmentsForLineno{equation}%
\patchBothAmsMathEnvironmentsForLineno{align}%
\patchBothAmsMathEnvironmentsForLineno{flalign}%
\patchBothAmsMathEnvironmentsForLineno{alignat}%
\patchBothAmsMathEnvironmentsForLineno{gather}%
\patchBothAmsMathEnvironmentsForLineno{multline}%
}

\theoremstyle{definition}

\newtheorem{definition}{Definition}[section]
\newtheorem{notice}{Remark}[section]
\newtheorem{corollary}{Corollary}[section]

\theoremstyle{remark}
\newtheorem{case}{Case}[section]

\theoremstyle{plain}
\newtheorem{theorem}{Theorem}[section]

\newtheorem{lemma}[theorem]{Lemma}

\begin{document}

\title[Compactness]{Compactness in Lorentz sequence spaces}

\author{Paweł Sawicki}

\address{Doctoral School of Quantum Information Theory, University of Gda\'nsk, 80-308 Gda\'nsk, Poland }

\email[P.~Sawicki]{p.a.sawicki@wp.pl}


\begin{abstract}
In this paper we are going to discuss compactness in Lorentz sequence spaces. Firstly, it will be shown how to define such a space, check whether a sequence belongs to it and calculate its norm. Equipped with this knowledge, we will proceed to propose usable compactness criteria for Lorentz sequence spaces, employing the concept of seminorms.
\end{abstract}

\maketitle

\section{Introduction}

Lorentz spaces were defined in year 1950 by George Gunter Lorentz in the paper \cite{lorentz} as functional spaces consisting of measurable functions on a bounded interval of $\mathbb{R}$. This definition can be naturally generalised to all possible spaces with measure, as in \cite{graf}. This source expands significantly on the known uses of Lorentz spaces and possibilities of further research.
\newline
\newline
As a special case, the space with measure can be chosen as $\mathbb{N}$ with the counting measure, thus defining a $l_{p, q}$ space. It is discussed in the article \cite{csdgvs}. A somewhat more general version, with a weight sequence $w$, is given by \cite{ciesielewicki}. Thus defined \emph{Lorentz sequence spaces} will be of main interest in our further considerations.
\newline
\newline
They can also be defined in a different but equivalent manner, through permutations, as in \cite{linden}. It is mentioned there that the concept of \emph{nonincreasing rearrangement} can be used to conduct the definition more smoothly. In the later sections here, it will be described precisely by the means of injective maps, serving the purpose of transforming a fully abstract definition into a constructive one.
\newline
\newline
To the best of our knowledge, there are no known compactness criteria in Lorentz spaces. In this paper, we are stating (and proving the correctness of) one for Lorentz sequence spaces.
\newline
\newline
Let us give a short reminder about compactness: a set in a metric space is called \emph{compact} if every its open cover has a finite subcover. A set is called \emph{relatively compact} when its closure is compact. A set is called \emph{precompact} if, for every $\varepsilon>0$, it has a finite cover by sets of radius at most $\varepsilon$. These two notions are equivalent when applied to subsets of a complete space. In metric spaces, a set is compact exactly when it is relatively compact and complete.
\newline
\newline
Being able to determine whether a set is (relatively) compact can be useful in a number of ways. For example, it enables methods efficient in solving nonlinear equations (Schauder fixed point theorem, Leray-Schauder degree), as well as investigation of compact operators, notable for their spectral properties.% And perhaps most famously, it is necessary for the exploration of compact embeddings between spaces, providing a powerful tool for partial differential equations.

\section{Basic definitions}

\begin{definition}[Space $l_p$]
For $a=(a_1, a_2, a_3...) \in \mathbb{R}^\mathbb{N}$ and $p \in [1, \infty)$ let $||a||_p=(\sum_{i=1}^\infty |a_i|^p)^{\frac{1}{p}}$. $||a||_p$ is known as a $p$-norm of a sequence. The space of all sequences with finite $p$-norm is known as $l_p$. It is a Banach space.
\end{definition}

\begin{definition}[Permutations of $\mathbb{N}$]
Let us denote the set of all permutations of $\mathbb{N}$ as $\Pi \subset \mathbb{N}^\mathbb{N}$. Thus, for any $\sigma = (\sigma_1, \sigma_2, \sigma_3, ...) \in \mathbb{N}^\mathbb{N}$, $\sigma \in \Pi$ if and only if the mapping defined by $\sigma$ (index to value) is bijective. Additionally, we can define a set $\Psi$ of all injective maps, so that $\Pi \subset \Psi \subset \mathbb{N}^\mathbb{N}$, and $\sigma \in \Psi$ if and only if the mapping defined by $\sigma$ is injective (one-to-one).
\end{definition}

\begin{definition}[Weight sequence]
A sequence $w$ of positive real numbers will be called a \emph{weight sequence} when it satisfies $w_1=1 \ge w_2 \ge w_3... \xrightarrow[\infty]{} 0$ and $\sum_{i=1}^\infty w_i=+\infty$.
\end{definition}

\begin{definition}[Lorentz space]
For any $a=(a_1, a_2, a_3, ...) \in \mathbb{R}^\mathbb{N}$, $p \in [1, \infty)$, weight sequence $w$ and permutation $\sigma$, let us define the sequence $L(a, p, w, \sigma)$: $L_i=a_{\sigma_i} \cdot \sqrt[p]w_i$, $i \in \mathbb{N}$. The Lorentz norm $||\cdot||_{p, w}$ is given as $$||a||_{p, w}=\sup_{\sigma \in \Pi}: ||L(a, p, w, \sigma)||_p.$$ All sequences with finite Lorentz $p, w$-norm form a Lorentz space $L_{p, w}$.
\end{definition}

\begin{theorem}
For any $a$, $p$, $w$ defined as above, $$||a||_{p, w}=\sup_{\sigma \in \Psi}: ||L(a, p, w, \sigma)||_p.$$
\end{theorem}
\begin{proof}
It is obvious that $\sup_{\sigma \in \Pi}: ||L(a, p, w, \sigma)||_p \le \sup_{\sigma \in \Psi}: ||L(a, p, w, \sigma)||_p$. We can concentrate on proving the reverse inequality, which can be rewritten as $\sup_{\sigma \in \Psi}: \sum_{i=1}^\infty |a_{\sigma_i}|^p \cdot w_i \le \sup_{\sigma \in \Pi}: \sum_{i=1}^\infty |a_{\sigma_i}|^p \cdot w_i$. By definition, the left-hand side equals $\sup_{\sigma \in \Psi}: \lim_{n \to \infty} \sum_{i=1}^n |a_{\sigma_i}|^p \cdot w_i$. However, for any $\psi \in \Psi$ and $n \in \mathbb{N}$, there exists a permutation $\psi_n \in \Pi$ such that $\psi_i=\psi_{n, i}$ for any $i \le n$. Thus, $$\forall \varepsilon >0 \; \exists \psi \in \Psi, m \in \mathbb{N}, \psi_m \in \Pi: \; \sup_{\sigma \in \Psi}: \lim_{n \to \infty} \sum_{i=1}^n |a_{\sigma_i}|^p \cdot w_i \le$$ $$\le \varepsilon + \lim_{n \to \infty} \sum_{i=1}^n |a_{\psi_i}|^p \cdot w_i \le 2 \varepsilon + \sum_{i=1}^m |a_{\psi_i}|^p \cdot w_i \le$$ $$\le 2 \varepsilon + \lim_{n \to \infty} \sum_{i=1}^n |a_{\psi_{m, i}}|^p \cdot w_i \le 2 \varepsilon + \sup_{\sigma \in \Pi}: \sum_{i=1}^\infty |a_{\sigma_i}|^p \cdot w_i.$$ Hence indeed $\sup_{\sigma \in \Psi}: \sum_{i=1}^\infty |a_{\sigma_i}|^p \cdot w_i \le \sup_{\sigma \in \Pi}: \sum_{i=1}^\infty |a_{\sigma_i}|^p \cdot w_i$ and this completes the proof. For divergent series ($a \notin L_{p, w}$), the reasoning is fully analogous.
\end{proof}

\begin{theorem}[Lorentz norm supremum is achieved]
\label{basicth2}
For any values of $a$, $p$, $w$ such that $a \in L_{p, w}$, there exists $\sigma \in \Psi$ such that $||a||_{p, w}=||L(a, p, w, \sigma)||_p$.
\end{theorem}
\begin{proof}
Initial observation: we can assume without loss of generality that the sequence $a$ has infinitely many nonzero entries, otherwise $\sigma$ permutates these nonzero entries and the problem is trivial.
\newline
\newline
First part: we will show that when there exists $\sigma \in \Psi$ for which $|a_\sigma|$ is nonincreasing and covers all nonzero elements of the sequence, for this $\sigma$ the function $L$ attains its maximum. For this purpose, let us choose any other injective map $\psi$. We need to show that $$||L(a, p, w, \sigma)||_p \ge ||L(a, p, w, \psi)||_p,$$ or equivalently $$\sum_{i=1}^\infty |a_{\sigma_i}|^p \cdot w_i \ge \sum_{i=1}^\infty |a_{\psi_i}|^p \cdot w_i.$$ For clarity of notation, we can pick a (nonincreasing positive) sequence $c$ by $c_i=|a_{\sigma_i}|^p$, with a corresponding map $\xi : \mathbb{N} \to \mathbb{N}$ such that $c_{\xi_i}=|a_{\psi_i}|^p$. Now we wish to prove that $\sum_{i=1}^\infty c_i \cdot w_i \ge \sum_{i=1}^\infty c_{\xi_i} \cdot w_i$. Because the sum of a series is defined as the limit of its partial sums, it is enough to show for any $n \in \mathbb{N}$ that $$\sum_{i=1}^n c_i \cdot w_i \ge \sum_{i=1}^n c_{\xi_i} \cdot w_i.$$ We can assume without loss of generality that $\xi$ maps $\{1, 2, ..., n\}$ onto $\{1, 2, ..., n\}$, as otherwise the right-hand side of our inequality could be trivially increased. Now the simplest proof is by the principle of mathematical induction. The first step (for $n=1$) is obvious, so let us assume that the lemma is true for $n-1$. Moreover, let $\xi$ map $l$ to $n$ and $n$ to $k$. We can define a new map $\xi'$ which differs from $\xi$ only by mapping $l$ to $k$ and $n$ to $n$. By the inductive assumption, $\sum_{i=1}^n c_i \cdot w_i \ge \sum_{i=1}^n c_{\xi'_i} \cdot w_i$. There also occurs $$\sum_{i=1}^n c_{\xi'_i} \cdot w_i-\sum_{i=1}^n c_{\xi_i} \cdot w_i=c_kw_l+c_nw_n-c_kw_n-c_nw_l=(c_k-c_n)(w_l-w_n) \ge 0.$$ Therefore $\sum_{i=1}^n c_i \cdot w_i \ge \sum_{i=1}^n c_{\xi_i} \cdot w_i$. This completes the first part of the proof.
\newline
\newline
Second part: we will show that when there is no $\sigma \in \Psi$ for which $|a_\sigma|$ would be nonincreasing and would contain all nonzero elements of the sequence $a$, such a sequence $a$ does not belong to any $L_{p, w}$. For this purpose, let us define the sequence $\sigma$ by the following conditions: $\mathbb{N}_i=\mathbb{N} \backslash \{ \sigma_1, \sigma_2, ..., \sigma_{i-1} \}$, $\sigma_i \in \mathbb{N}_i$, $\forall j \in \mathbb{N}_i: \; |a_{\sigma_i}| \ge |a_j| \land (|a_{\sigma_i}|=|a_j| \implies \sigma_i \le j)$. Now try to assume that these conditions do not define a unique $\sigma_i$ (and does define unique $\sigma_h$ for all natural numbers $h<i$). Then, by contraposition, $$\forall j \in \mathbb{N}_i \; \exists m \in \mathbb{N}_i: \; |a_m| \ge |a_j| \land m>j.$$ Therefore it would be possible to construct an infinite sequence in $\mathbb{N}_i$ $j_1, j_2, j_3...$ such that $j_k$ would be strictly increasing and $|a_{j_k}|$ would be increasing. However, by the initial observation, it is possible to choose such $j_1 \in \mathbb{N}_i$ that $|a_{j_1}|>0$. Thus we could define $\psi \in \Psi = (j_1, j_2, j_3, ...)$, and for any weight sequence $w$ and $p \in [1, \infty)$ the series $|a_{\psi_i}|^p \cdot w_i$ would be divergent (because the left-hand factors are greater or equal $|a_{j_1}|^p>0$ and the series of the right-hand factors is divergent), showing that in such a case $a \notin L_{p, w}$. Hence we have proven that, for $a \in L_{p, w}$, our definition of the injective map $\sigma \in \Psi$ is correct and renders a nonincreasing $|a_\sigma|$, realising the Lorentz norm supremum.
\end{proof}

\begin{notice}[Behaviour of $|a_\sigma|$]
When $\lim_{n \to \infty} |a_{\sigma_n}|^p>0$, $a \notin L_{p, w}$. When the limit equals $0$ and the series is convergent, $a \in L_{p, w}$. When the limit equals $0$ but the series is divergent, $a$ might or might not belong to $L_{p, w}$.
\end{notice}

\begin{case}
For $p=1$, $|a_{\sigma_i}|=w_i=\frac{1}{i}$ there occurs $a \in L_{p, w}$, $||a||_{p, w}=\frac{\pi^2}{6}$.
\end{case}

\begin{case}
For $p=1$, $|a_{\sigma_i}|=w_i=\frac{1}{\sqrt{i}}$ there occurs $a \notin L_{p, w}$.
\end{case}

\section{Compactness}

For determining compactness criteria in Lorentz sequence spaces, we will use the seminorm approach, as laid out in \cite{gkm}.

\begin{definition}[Seminorm]
For any real vector space $V$, a \emph{seminorm} is a function $s: V \rightarrow [0, \infty)$ such that $\forall x \in \mathbb{R}, v \in V \; s(x \cdot v)=|x| \cdot s(v)$, $\forall v, w \in V \; s(v)+s(w) \ge s(v+w)$. A seminorm for which $s(v)=0 \iff v=0$ is called a \emph{norm}, as used in the previous section.
\end{definition}

The following two definitions are repeated from the aforementioned article.

\begin{definition}[Regular family of seminorms]
\label{compdef2}
When on a normed space $E, ||\cdot||$ there is a family of seminorms $\{||\cdot||_i\}_{i \in I}$, it is known as \emph{regular} if it satisfies three conditions:

1) $\forall x \in E \; ||x||=\sup_{i \in I} \; ||x||_i$;

2) the family $\{||\cdot||_i\}_{i \in I}$ forms a directed set, i.e. $$\forall i, j \in I \; \exists k \in I \; \forall x \in E \; ||x||_k \ge ||x||_i \land ||x||_k \ge ||x||_j;$$

3) every bounded sequence in $E$ has a subsequence which is Cauchy with respect to all seminorms in the family.
\end{definition}

\begin{definition}[Equinormed set]
\label{compdef3}
A set $A \subset E$ is \emph{equinormed} (with respect to a given family of seminorms indexed by set $I$) if $$\forall \varepsilon >0 \; \exists i \in I \; \forall x \in A \; ||x|| \le ||x||_i+\varepsilon.$$
\end{definition}

\begin{notice}
\label{comprem1}
For any positive $p$, a bounded set $A \in E$ is equinormed if and only if $$\forall \varepsilon >0 \; \exists i \in I \; \forall x \in A \; ||x||^p \le ||x||_i^p+\varepsilon.$$ This follows from the fact that the power function $x \mapsto x^p$ is uniformly continuous on bounded sets.
\end{notice}

The following theorem is given in \cite{gkm} as Theorem 3.19.

\begin{theorem}
\label{compth1}
For any regular family of seminorms, a non-empty set $A \subset E$ is precompact when and only when $A-A$ is equinormed and $A$ is bounded.
\end{theorem}

Now we need to define a regular family of seminorms for a Lorentz space $L_{p, w}$. This family will be indexed by $i \in \mathbb{N}$ and defined as $||a||_{p, w, i}=||r_i(a)||_{p, w}$, where $r_i: L_{p, w} \to L_{p, w}$ is a linear projection given by $$r_i(a)=(a_1, a_2, ..., a_i, 0, 0, ...).$$

\begin{notice}
$r_i(L_{p, w}) \subset L_{p, w}$ is a subset closed under addition and scalar multiplication, and there occurs $a \in r_i(L_{p, w}) \implies ||a||_{p, w}=||a||_{p, w, i}$. Herefrom it follows instantly that $||\cdot||_{p, w, i}$ is a seminorm. Its value is achieved on some permutation of the set $\{1,...,i\}$.
\end{notice}

At this point it is necessary to show that this family is regular.

\begin{lemma}[Second condition of regularity]
For any appropriate $a$, $p$, $w$, $n$ there occurs $||a||_{p, w, n} \le ||a||_{p, w, n+1}$ and hence the family $\{||\cdot||_{p, w, i}\}_{i \in \mathbb{N}}$ fulfils the second condition of regularity.
\end{lemma}
\begin{proof}
Let $\sigma$ be a permutation which achieves the supremum for $r_n(a)$ and $\psi$ for $r_{n+1}(a)$. $\sigma$ and $\psi$ can be chosen so that $\sigma_{n+1}=n+1$, $\sigma(\{1,...,n\})=\{1,...,n\}$ and $\psi(\{1,...,n+1\})=\{1,...,n+1\}$. Then it suffices to show that $$\sum_{i=1}^\infty |r_n(a)_{\sigma_i}|^p \cdot w_i \le \sum_{i=1}^\infty |r_{n+1}(a)_{\psi_i}|^p \cdot w_i.$$ And indeed, $$\sum_{i=1}^\infty |r_n(a)_{\sigma_i}|^p \cdot w_i = \sum_{i=1}^n |a_{\sigma_i}|^p \cdot w_i \le \sum_{i=1}^{n+1} |a_{\sigma_i}|^p \cdot w_i \le \sum_{i=1}^{n+1} |a_{\psi_i}|^p \cdot w_i = \sum_{i=1}^\infty |r_{n+1}(a)_{\psi_i}|^p \cdot w_i.$$
\end{proof}

\begin{lemma}[First condition of regularity]
For any appropriate $a$, $p$, $w$, $n$ there occurs $||a||_{p, w}=\lim_{n \to \infty} ||a||_{p, w, n}$.
\end{lemma}
\begin{proof}
Let $\sigma$ be an injective map which achieves the supremum for $a$ and $\sigma^{(n)}$ for $r_{n}(a)$. We need to prove that $$\lim_{n \to \infty} \sum_{i=1}^n |a_{\sigma_i^{(n)}}|^p \cdot w_i = \sum_{i=1}^\infty |a_{\sigma_i}|^p \cdot w_i.$$ Let $\chi$ denote the characteristic function (equal $1$ if the expression in brackets is true and $0$ otherwise). Obviously, $$\lim_{n \to \infty} \sum_{i=1}^n |a_{\sigma_i^{(n)}}|^p \cdot w_i \ge \lim_{n \to \infty} \sum_{i=1}^n |a_{\sigma_i}|^p \cdot w_i \cdot \chi(\sigma_i \le n) = \sum_{i=1}^\infty |a_{\sigma_i}|^p \cdot w_i.$$ On the other hand, $$\forall n \in \mathbb{N} \; \sum_{i=1}^n |a_{\sigma_i^{(n)}}|^p \cdot w_i \le \sum_{i=1}^\infty |a_{\sigma_i^{(n)}}|^p \cdot w_i \le \sum_{i=1}^\infty |a_{\sigma_i}|^p \cdot w_i.$$ Since inequality is preserved in the limit, this completes the proof.
\end{proof}

\begin{lemma}[Third condition of regularity]
Every bounded sequence in $L_{p, w}$ has a subsequence that is Cauchy for all seminorms.
\end{lemma}
\begin{proof}
As a minor lemma, let us prove the thesis for a particular seminorm $||\cdot||_{p, w, n}$, for a fixed $n \in \mathbb{N}$. We consider a sequence $(a_1, a_2, a_3, ...)$ where each element $a_i = (a_{i, 1}, a_{i, 2}, a_{i, 3}, ...) \in L_{p, w}$. Let $\sigma^{(i)}$ denote the permutation with which the supremum for $r_n(a_i)$ is attained. As we can see, $\sigma^{(i)}(\{1,...,n\})=\{1,...,n\}$ and we may associate $a_i$ with a permutation of the finite set $\{1,...,n\}$. Since there are finitely many (namely $n!$) such permutations, one of them must be chosen for infinitely many elements of the sequence $(a_i)$. This subsequence will be denoted as $(a_{n_1}, a_{n_2}, a_{n_3}, ...)$ and we can write a universal formula $||a_{n_j}||_{p, w, n}=(\sum_{i=1}^n |a_{n_j, \sigma_i}|^p \cdot w_i)^{\frac{1}{p}}$. Now this can be interpreted as a norm in $\mathbb{R}^n$, and because all norms in $\mathbb{R}^n$ are equivalent, it has the desired property that a bounded sequence has a Cauchy subsequence. Being bounded in $||\cdot||_{p, w}$ implies being bounded in all $||\cdot||_{p, w, n}$, so the proof of the lemma is finished.
\newline
\newline
Now it is possible to use the diagonal argument to extend the minor lemma to all seminorms. From any bounded sequence $(a_1, a_2, a_3, ...)$ in $L_{p, w}$ we can choose a subsequence $(a_{b_1}, a_{b_2}, a_{b_3}, ...)$ which is Cauchy with respect to $||\cdot||_{p, w, 1}$, from it a subsubsequence $(a_{b_{c_1}}, a_{b_{c_2}}, a_{b_{c_3}}, ...)$ which is Cauchy also with respect to $||\cdot||_{p, w, 2}$, and proceed so \emph{ad infinitum}. The cross-subsequence $(a_1, a_{b_2}, a_{b_{c_3}}, ...)$ fulfils the condition.
\end{proof}

\begin{notice}[Equinormedness]
With our family of seminorms, Definition \ref{compdef3} can be trivially interpreted as stating that a set $A \in L_{p, w}$ is equinormed when $\lim_{i \to \infty} ||a||_{p, w, i}=||a||_{p, w}$ uniformly on $A$. And equivalently, by Remark \ref{comprem1}, when $\lim_{i \to \infty} ||a||_{p, w, i}^p=||a||_{p, w}^p$ uniformly on $A$.
\end{notice}

\begin{corollary}
A set $A \subset L_{p, w}$ is precompact if and only if it is bounded in $||\cdot||_{p, w}$ and the set $A-A$ is equinormed with respect to $||\cdot||_{p, w, i}$. This is a direct consequence of Theorem \ref{compth1}.
\end{corollary}

\section{Main result}

Now we have the necessary tools to begin proving the main theorem of this article, namely that a set $A \subset L_{p, w}$ is precompact if and only if it is bounded in $||\cdot||_{p, w}$ and equinormed with respect to $||\cdot||_{p, w, i}$. Therefore, the core of the proof is to show that the equinormedness of $A$ implies the equinormedness of $A-A$ (the reverse is much easier). This will be done through a series of lemmas, showing the uniform properties of $A$ and acquiring other interesting observations in the meantime.

\begin{lemma}[Sequences $a \in A$ have a bounded number of element repetitions]
\label{mainlem1}
If $A$ is bounded, then for any $d>0$ there is such a number $\lambda(d)$ that no sequence in $A$ has more than $\lambda(d)$ elements greater or equal $d$ in their absolute values, or in other words $$\forall d>0 \; \exists \lambda(d) \in \mathbb{N} \; \forall a \in A \; |a_{\sigma_{\lambda(d)+1}}|<d,$$ where, by our previous notation, $\sigma$ is a nonincreasing natural injective map.
\end{lemma}
\begin{proof}
Since $A$ is bounded, let $M$ be its upper bound in $||\cdot||_{p, w}$. Trying to assume that for some $d$ there is no such upper bound $\lambda(d)$ on the amount of repetitions, we get $$\forall n \in \mathbb{N} \; \exists a \in A \; M \ge ||a||_{p, w} \ge d^p \cdot \sum_{i=1}^n w_i.$$ Hence for any $n$ there would occur $\frac{M}{d^p} \ge \sum_{i=1}^n w_i$, and yet we know that the series of the weight sequence $w$ is divergent. This proves the lemma by contradiction.
\end{proof}

For clarity of notation, with a pre-determined weight sequence $w$ and power $p$, let us denote $||a||_{p, w, i}^p$ as $S_i(a)$. When $A$ is equinormed with respect to $||\cdot||_{p, w, i}$, it means that $$\forall \varepsilon>0 \; \exists n(\varepsilon) \in \mathbb{N} \; \forall a \in A \; ||a||_{p, w}^p-S_{n(\varepsilon)}(a)<\varepsilon.$$ Hereforth, by convention, we will denote by $n(\varepsilon)$ the lowest natural number having said property.
\newline
\newline
In the following lemma, we will split our hypothesis into two possible cases and prove for one of them that $A-A$ is equinormed, so that we will be able to use the other case as an assumption to continue the proof.

\begin{lemma}
\label{mainlem2}
If $\exists m \in \mathbb{N} \; \forall \varepsilon >0 \; n(\varepsilon) \le m$, then $A-A$ is equinormed.
\end{lemma}
\begin{proof}
By the definition of $n(\varepsilon)$, this condition means that $$\exists n \in \mathbb{N} \; \forall a \in A \; ||a||_{p, w}^p=S_n(a).$$ This occurs if and only if $a_i=0$ for any $i>n$, because any nonzero values would increase the norm without impacting the seminorm. However, in this case the set $A-A$ holds the same property, so it is equinormed.
\end{proof}

Starting from now on, we will be able to use the assumption that $\forall m \in \mathbb{N} \; \exists \varepsilon >0 \; n(\varepsilon)>m$ (the index of seminorms increases as $\varepsilon$ approaches $0$).

\begin{lemma}[Sequences $a \in A$ approach $0$ uniformly]
\label{mainlem3}
If $A$ is bounded and equinormed, then $$\forall d>0 \; \exists N \in \mathbb{N} \; \forall a \in A \; \forall n \ge N \; |a_n|<d.$$
\end{lemma}
\begin{proof}
As usual, we will conduct an indirect proof. Let us assume that $$\exists d>0 \; \forall N \in \mathbb{N} \; \exists a \in A \; \exists n \ge N \; |a_n| \ge d$$ in order to show that $A$ would not be equinormed. For this purpose, we want to estimate $||a||_{p, w}^p-S_k(a)$ for all values of $k \in \mathbb{N}$ and show that it cannot be arbitralily small. By Lemmas \ref{mainlem1}, \ref{mainlem2}, on assumption of the boundedness and equinormedness of $A$, it is possible to check only the values $k>\lambda(\frac{d}{2})$ whenever the expected difference $\varepsilon>0$ is small enough that $n(\varepsilon)>\lambda(\frac{d}{2})$. By the indirect assumption, we could choose a sequence $a \in A$ which has at least one element $|a_n| \ge d$ with $n \ge k$. Denoting by $k'< \lambda(d)$ the number of elements greater or equal $d$ among $\{|a_1|, |a_2|, ..., |a_k|\}$, and using the previous notation of $\sigma^{(k)}$ for a permutation which achieves the supremum for $r_k(a)$, there occurs $$||a||_{p, w}^p \ge |a_{\sigma_1^{(k)}}|^p w_1+|a_{\sigma_2^{(k)}}|^p w_2+...+|a_{\sigma_{k'}^{(k)}}|^p w_{k'}+d^p w_{k'+1}+|a_{\sigma_{k'+1}^{(k)}}|^p w_{k'+2}+...+|a_{\sigma_{k-1}^{(k)}}|^p w_k.$$ Thus, it is possible to lay out the estimation $$||a||_{p, w}^p-S_k(a) \ge (d^p-|a_{\sigma_{k'+1}^{(k)}}|^p) \cdot w_{k'+1}+(|a_{\sigma_{k'+1}^{(k)}}|^p-|a_{\sigma_{k'+2}^{(k)}}|^p) \cdot w_{k'+2}+...$$ $$+(|a_{\sigma_{\lambda(\frac{d}{2})}^{(k)}}|^p-|a_{\sigma_{\lambda(\frac{d}{2})+1}^{(k)}}|^p) \cdot w_{\lambda(\frac{d}{2})+1} \ge (d-\frac{d}{2^p}) \cdot w_{\lambda(\frac{d}{2})+1}.$$ It suffices to choose such $\varepsilon < (d-\frac{d}{2^p}) \cdot w_{\lambda(\frac{d}{2})+1}$ that $n(\varepsilon)>\lambda(\frac{d}{2})$ to see that for any appropriate $k$ and some $a \in A$ there is $||a||_{p, w}^p-S_k(a)> \varepsilon$. For this reason, $A$ cannot be equinormed after all, and it proves the lemma by contradiction.
\end{proof}

\begin{notice}
Lemma \ref{mainlem3} is a rather interesting result on its own. To see it, let us define a function $\gamma: \mathbb{R} \rightarrow \mathbb{N}$ where $\gamma(d)=N$ (the minimum value satisfying the condition from Lemma \ref{mainlem3}). As it turns out, its naturally defined reverse sequence $\gamma^{-1}(n)=\sup \{x \in \mathbb{R}: \gamma(x) \ge n\}$ is a majorizing sequence for all $a \in A$. However, there is no guarantee that $\gamma^{-1} \in L_{p, w}$.
\end{notice}

In order to proceed with the subsequent lemmas, it will be reasonable to introduce, for given values of $p$ and $w$, a few more convenient symbols similar to $S_i(a)$ for seminorms. For a given $i$, the sequence $a \in A \subset L_{p, w}$ can be naturally divided into its head $a_{(1...i)}$ and its tail $a_{(i+1...\infty)}=(a_{n+i})_{n \in \mathbb{N}}$. Now we can define the following concepts:

\begin{definition}[Tail complement of the seminorm]
\label{maindef1}
$$\tilde{S}_i(a)=||a_{(i+1...\infty)}||_{p, w_{(i+1...\infty)}}^p$$
This formula can be naturally described as \emph{tail complement of the seminorm} because $S_i(a)+\tilde{S}_i(a)$ is a certain reconfiguration of the sequence $a$ multiplied piecewise by the weight sequence $w$.
\end{definition}

\begin{definition}[Inverse function of $\sigma$]
For $\sigma \in \Psi$, we will use $\sigma^{-1}$ in its natural sense, as the inverse function of the map $\sigma$. By slight abuse of notation, for any indices $i \in \mathbb{N} \setminus \sigma(\mathbb{N})$ we assign $w_{\sigma_i^{-1}}=0$. Thereupon, any sum over $\sigma^{-1}(\mathbb{N})$ can be written without additional constraints, as any element outside the domain would be equal $0$ and thus would not influence the sum.
\end{definition}

The definitions given below, just like Definition \ref{maindef1}, are stated for a fixed $a \in L_{p, w}$. Whenever they refer to an injective map $\sigma$, it means the optimal map existing for $a$ by Theorem \ref{basicth2}, and "norm formula" means the formula defining the $p$-th power of the Lorentz norm.

\begin{definition}[Head of the norm]
$$H_i(a)=\sum_{j=1}^i |a_j|^p \cdot w_{\sigma_j^{-1}}$$
This is the part of the norm formula for sequence $a$ counted over its first $i$ elements.
\end{definition}

\begin{definition}[Tail of the norm]
$$\tilde{H}_i(a)=\sum_{j=i+1}^\infty |a_j|^p \cdot w_{\sigma_j^{-1}}$$
This is the other part of the norm formula. There occurs $H_i(a)+\tilde{H}_i(a)=||a||_{p, w}^p$.
\end{definition}

\begin{definition}[Head by weight of the norm]
$$W_i(a)=\sum_{j=1}^i |a_{\sigma_j}|^p \cdot w_j$$
This is the part of the norm formula for sequence $a$ counted over its first $i$ elements, ordered by weight.
\end{definition}

\begin{definition}[Tail by weight of the norm]
$$\tilde{W}_i(a)=\sum_{j=i+1}^\infty |a_{\sigma_j}|^p \cdot w_j$$
This is, once again, the other part of the norm formula. There occurs $W_i(a)+\tilde{W}_i(a)=||a||_{p, w}^p$.
\end{definition}

\begin{definition}[Norm of the tail]
$$T_i(a)=||a_{(i+1...\infty)}||_{p, w}^p$$
This differs from the tail complement of the seminorm in that all values from $w$ can be used again.
\end{definition}

\begin{lemma}[Relationships between the defined terms]
\label{mainlem4}
$$S_i(a)+\tilde{S}_i(a) \le ||a||_{p, w}^p=H_i(a)+\tilde{H}_i(a)=W_i(a)+\tilde{W}_i(a) \le S_i(a)+T_i(a)$$
\end{lemma}
\begin{proof}
The first inequality follows directly from the fact that $S_i(a)+\tilde{S}_i(a)$ cannot be greater than the $p$-th power of the Lorentz norm, which is realised by $H_i(a)+\tilde{H}_i(a)$. In the second inequality, analogously, there is $H_i(a) \le S_i(a)$ by definition of the seminorm, so we only need to show that $\tilde{H}_i(a) \le T_i(a)$. For this purpose, let us concentrate on any particular summand associated with $a_j$ ($j>i$). In $\tilde{H}_i(a)$, it takes the form of $|a_j|^p \cdot w_{\sigma_j^{-1}}$. In $T_i(a)$, it can be similarly described as $|a_j|^p \cdot w_{\psi_j^{-1}}$, with the difference that while the map $\sigma_j^{-1}$ is equal to the number of of all elements of $a$ greater than $|a_j|$, the map $\psi_j^{-1}$ counts only over the elements that have indices above $i$. Hence $\sigma_j^{-1} \ge \psi_j^{-1}$, $w_{\sigma_j^{-1}} \le w_{\psi_j^{-1}}$ and finally $\tilde{H}_i(a) \le T_i(a)$, completing the proof of the lemma.
\end{proof}

\begin{notice}
\label{mainrem2}
Recombining our inequality yields $||a||_{p, w}^p - S_i(a) \le T_i(a)$. Therefore, when in any set $A$ the norms of the tails approach $0$ uniformly for $i \rightarrow \infty$, said set is equinormed.
\end{notice}

\begin{lemma}[Uniform behaviour of norms of the tails]
\label{mainlem5}
When norms of the tails in a set $A \subset L_{p, w}$ approach $0$ uniformly for $i \rightarrow \infty$, the same is true for $A-A$.
\end{lemma}
\begin{proof}
We can use the norm properties of $||\cdot||_{p, w}$, applying the triangle inequality to claim instantly that $$\forall \varepsilon>0 \; \exists k \in \mathbb{N} \; \forall a, b \in A \; ||a_{(k+1...\infty)}||_{p, w}<\varepsilon, \; ||b_{(k+1...\infty)}||_{p, w}<\varepsilon \implies$$ $$\implies ||(a-b)_{(k+1...\infty)}||_{p, w}<2 \cdot \varepsilon.$$ Therefore $$\forall \varepsilon>0 \; \exists k \in \mathbb{N} \; \forall a, b \in A \; ||a_{(k+1...\infty)}||_{p, w}^p<\varepsilon, \; ||b_{(k+1...\infty)}||_{p, w}^p<\varepsilon \implies$$ $$\implies ||(a-b)_{(k+1...\infty)}||_{p, w}<2 \cdot \sqrt[p]{\varepsilon} \iff ||(a-b)_{(k+1...\infty)}||_{p, w}^p<2^p \cdot \varepsilon.$$
\end{proof}

These results give us a clear path towards proving the entire theorem. Now it is enough to show that equinormedness of $A$ implies that norms of the tails in $A$ approach $0$ uniformly for $i \rightarrow \infty$. There is an intuitive and fairly trivial step of showing the same for tails by weight of the norms in $A$, and making these two steps will be the content of our final two lemmas.

\begin{lemma}[Tails by weight of the norms approach $0$ uniformly]
When $A$ is bounded and equinormed,
$$\forall \varepsilon>0 \; \exists N \in \mathbb{N} \; \forall a \in A \; \forall i \ge N \; \tilde{W}_i(a)=\sum_{j=i+1}^\infty |a_{\sigma_j}|^p \cdot w_j < \varepsilon.$$ (Reading this notation, it is important to remember that the map $\sigma$ depends on the choice of $a$.)
\end{lemma}
\begin{proof}
Because $W_i(a)+\tilde{W}_i(a)=||a||_{p, w}^p$, and the $p$-th powers of seminorms approach $||a||_{p, w}$ uniformly on $A$, it is enough to show that $\forall a \in A \; \forall i \in \mathbb{N} \; S_i(a) \le W_i(a)$. And we have already proven this, given that $W_i(a)$ is constructed according to the procedure described in Theorem \ref{basicth2}.
\end{proof}

\begin{lemma}[Norms of the tails approach $0$ uniformly]
\label{mainlem7}
When $A$ is bounded and equinormed,
$$\forall \varepsilon>0 \; \exists N \in \mathbb{N} \; \forall a \in A \; \forall i \ge N \; T_i(a)=||a_{(i+1...\infty)}||_{p, w}^p < \varepsilon.$$
\end{lemma}
\begin{proof}
By the previous lemma, there exists such $k \in \mathbb{N}$ that $\forall a \in A \; \tilde{W}_k(a) < \frac{\varepsilon}{2}$. Now let $b$ be the $i$-th tail of some $a \in A$, that is $b=a_{(i+1...\infty)}$. It is easy to notice that also $\tilde{W}_k(b) < \frac{\varepsilon}{2}$. Indeed, $$\tilde{W}_k(a)-\tilde{W}_k(b)=\sum_{j=k+1}^\infty w_j(|a_{\sigma_j}|^p-|b_{\psi_j}|^p),$$ where $\psi$ is a map such that $|b_\psi|$ is nonincreasing (hence it is the optimal injective map for the sequence $b$). Yet, $b$ is a subsequence of $a$, so $b_\psi$ is a subsequence of the nonincreasing $|a_\sigma|$. Therefore, for all natural $j$, $|a_{\sigma_j}|^p-|b_{\psi_j}|^p>0$ and every summand in the expression is nonnegative.
\newline
\newline
At the same time, $T_i(a)=W_k(b)+\tilde{W}_k(b)$. Now we only need to find such $i$ that $W_k(b) < \frac{\varepsilon}{2}$. Since $b=a_{(i+1...\infty)}$, applying Lemma \ref{mainlem3} shows that for any $d>0$ choosing $N=\gamma(d^\frac{1}{p})$ returns $$\forall j \in \mathbb{N}, \; i \ge N, \; b \in A_{(i+1...\infty)} \; |b_{\psi_j}|^p<d.$$ Because $W_k(b)=\sum_{j=1}^k w_j \cdot |b_{\psi_j}|^p$, $$N=\gamma((\frac{\varepsilon}{2 \sum_{j=1}^k w_j})^\frac{1}{p})$$ gives us the desired result.
\end{proof}
We have carried out all the preparations necessary to prove the main theorem in a few short steps.

\begin{theorem}
A set $A \subset L_{p, w}$ is precompact if and only if it is bounded and equinormed.
\end{theorem}
\begin{proof}
When $A$ is bounded and equinormed, by Lemma \ref{mainlem7} norms of the tails in $A$ approach $0$ uniformly, so by Lemma \ref{mainlem5} the same is true for the tails in $A-A$. Hence, Remark \ref{mainrem2} shows that $A-A$ is equinormed, and we can claim that $A$ is precompact thanks to Theorem \ref{compth1}. On the other hand, it is well known that a precompact set is bounded. The fact that it is equinormed, with respect to a regular family of seminorms (in its normed space as per Definition \ref{compdef2}), is given in \cite{gkm}, as Proposition 3.13 states directly that every precompact subset is equinormed.
\end{proof}

On a final note, we will provide an alternative formulation of our compactness condition, which is ostensibly stronger but might be easier to use in particular cases. The proof will reuse many techniques from the previous two lemmas.

\begin{theorem}[Tails of the norms approach $0$ uniformly]
\label{mainth9}
When $A$ is bounded, it is equinormed if and only if
$$\forall \varepsilon>0 \; \exists N \in \mathbb{N} \; \forall a \in A \; \forall i \ge N \; \tilde{H}_i(a)=\sum_{j=i+1}^\infty |a_j|^p \cdot w_{\sigma_j^{-1}} < \varepsilon.$$
\end{theorem}
\begin{proof}
The implication from the stated fact to the equinormedness is rather simple. Since $H_i(a)+\tilde{H}_i(a)=||a||_{p, w}^p$, the values of $H_i(a)$ approach $||a||_{p, w}^p$ uniformly on $A$, and as we have mentioned in the proof of Lemma \ref{mainlem4}, there is $H_i(a) \le S_i(a)$ by definition of the seminorm, so (using Remark \ref{comprem1}) the seminorms also exhibit the same behaviour on $A$. We can concentrate on proving the implication in the opposite direction.
\newline
\newline
Similarly, because $H_i(a)+\tilde{H}_i(a)=||a||_{p, w}^p$, and seminorms approach $||a||_{p, w}^p$ uniformly on $A$, it is enough to show that $$\forall \varepsilon>0 \; \forall k \in \mathbb{N} \; \exists N \in \mathbb{N} \; \forall a \in A \; \forall i \ge N \; S_k(a)-H_i(a) < \varepsilon.$$ Now, let us try to estimate the expression $S_k(a)-H_i(a)$, once again breaking it down by elements. For any summand of $S_k(a)$ denoted as $w_l \cdot |a_{\sigma_l^{(k)}}|^p$ ($1 \le l \le k$), there are two possibilities: its unique counterpart $w_l \cdot |a_{\sigma_l}|^p$ occurs as a summand in $H_i(a)$ or it does not, the second type happening exactly when $\sigma_l>i$. Let us recall that $|a_{\sigma_l^{(k)}}|$ can be interpreted as the $l$-th greatest absolute value taken from the first $k$ elements of sequence $a$, while $|a_{\sigma_l}|$ is the $l$-th greatest absolute value from the whole sequence, so obviously $|a_{\sigma_l^{(k)}}| \le |a_{\sigma_l}|$.
\newline
\newline
For this reason, the summand differences of the first type can be estimated from above by $w_l \cdot |a_{\sigma_l^{(k)}}|^p-w_l \cdot |a_{\sigma_l}|^p \le 0$, and the summands of the second type by $w_l \cdot |a_{\sigma_l^{(k)}}|^p \le w_l \cdot |a_{\sigma_l}|^p$. This breakdown of the expression may ignore some summands from $H_i(a)$, but they are nonnegative and, in subtraction, can also be estimated from above by $0$. Finally, we are able to write the inequality $$S_k(a)-H_i(a) \le \sum_{j=1}^k w_j \cdot |a_{\sigma_j}|^p \cdot \chi(\sigma_j>i),$$ with $\chi$, like earlier, denoting the characteristic function (equal $1$ if the expression in brackets is true and $0$ otherwise). Now it is enough to apply Lemma \ref{mainlem3}, showing that for any $d>0$ it is possible to choose $N=\gamma(d^\frac{1}{p})$ and guarantee that $$\forall j \in \mathbb{N}, \; i \ge N, \; a \in A \; |a_{\sigma_j}|^p \cdot \chi(\sigma_j>i) < d.$$ This gives us $S_k(a)-H_i(a) < \sum_{j=1}^k w_j \cdot d$, yielding the thesis of this theorem for $$N=\gamma((\frac{\varepsilon}{\sum_{j=1}^k w_j})^\frac{1}{p}).$$
\end{proof}

\begin{corollary}
A set $A \subset L_{p, w}$ is precompact if and only if it is bounded in $||\cdot||_{p, w}$ and $$\forall \varepsilon>0 \; \exists N \in \mathbb{N} \; \forall a \in A \; \forall i \ge N \; \sum_{j=i+1}^\infty |a_j|^p \cdot w_{\sigma_j^{-1}} < \varepsilon.$$ This is a direct consequence of Theorem \ref{mainth9}.
\end{corollary}

{\bf Acknowledgements.} The author would like to thank Professor Jacek Gulgowski for his invaluable support through all the research and editorial process, particularly concerning the idea of investigating compactness with use of seminorms.

\end{document}